\documentclass[reqno, 12pt,a4letter]{amsart}
\usepackage{amsmath,amsxtra,amssymb,latexsym, amscd,amsthm}
\usepackage[utf8]{inputenc}
\usepackage[mathscr]{euscript}
\usepackage{mathrsfs}
\usepackage[english]{babel}
\usepackage{color}
\usepackage{cite}

\def\rank{\mathrm{rank}}
\def\covol{\mathrm{Covol}}
\def\vol{\mathrm{Vol}}
\DeclareMathOperator{\Conv}{Conv}
\DeclareMathOperator{\Supp}{Supp}

\newtheorem {theorem}{Theorem}[section]
\newtheorem {corollary}{Corollary}[section]
\newtheorem {lemma}{Lemma}[section]

\newtheorem {definition}{Definition}[section]
\newtheorem {remark}{Remark}[section]

\def\ar{a\kern-.370em\raise.16ex\hbox{\char95\kern-0.53ex\char'47}\kern.05em}
\def\ees{{\accent"5E e}\kern-.385em\raise.2ex\hbox{\char'23}\kern-.08em}
\def\eex{{\accent"5E e}\kern-.470em\raise.3ex\hbox{\char'176}}
\def\AR{A\kern-.46em\raise.80ex\hbox{\char95\kern-0.53ex\char'47}\kern.13em}
\def\EES{{\accent"5E E}\kern-.5em\raise.8ex\hbox{\char'23 }}
\def\EEX{{\accent"5E E}\kern-.60em\raise.9ex\hbox{\char'176}\kern.1em}
\def\ow{o\kern-.42em\raise.82ex\hbox{
  \vrule width .12em height .0ex depth .075ex \kern-0.16em \char'56}\kern-.07em}
\def\OW{O\kern-.460em\raise1.36ex\hbox{
\vrule width .13em height .0ex depth .075ex \kern-0.16em \char'56}\kern-.07em}
\def\UW{U\kern-.42em\raise1.36ex\hbox{
\vrule width .13em height .0ex depth .075ex \kern-0.16em \char'56}\kern-.07em}
\def\DD{D\kern-.7em\raise0.4ex\hbox{\char '55}\kern.33em}

\title[Milnor number of complete intersection]{Uniform stable radius and Milnor number for non-degenerate isolated complete intersection singularities}

\author{TAT THANG NGUYEN$^{\dag}$}
\address{$^\dag$Institute of Mathematics, 18, Hoang Quoc Viet Road, Cau Giay District 10307, Hanoi, Vietnam}
\email{ntthang@math.ac.vn}

\subjclass{14D05,14D06,14B05,14B07}

\keywords{isolated complete intersection singularity, Milnor fibration, Milnor number, mixed newton number, nondegenerate,  uniform stable radius}

\date{ \today}

\begin{document}
\maketitle

\begin{abstract}
We prove that for two germs of analytic mappings $f,g\colon (\mathbb{C}^n,0) \rightarrow (\mathbb{C}^p,0)$  with the same  Newton polyhedra  which are (Khovanskii) non-degenerate and their zero sets  are  complete intersections with isolated singularity at the origin, there is a piecewise analytic family $\{f_t\}$ of analytic maps with $f_0=f, f_1=g$ which has a so-called {\it uniform stable radius for the Milnor fibration}. As a corollary, we show that their Milnor numbers are equal. Also, a formula for the Milnor number  is given in terms of the Newton polyhedra of the component functions. This is a generalization of the result by   
C. Bivia-Ausina. Consequently, we obtain that the Milnor number of a non-degenerate isolated complete intersection singularity is an invariance of Newton boundaries.
\end{abstract}

\section{Introduction}
Let $f \colon (\mathbb{C}^n,0) \rightarrow (\mathbb{C}^p,0)$ be an analytic mapping germ such that $V:=f^{-1}(0)$ is a complete intersection with isolated singularity at the origin. Let $\epsilon_0$ be a positive and
sufficiently small real number such that the sphere $\Bbb{S}^{2n-1}_{\epsilon}$ intersects the variety $V$ transversally for all $\epsilon\leq \epsilon_0$. Let $U$ be an open neighbourhood  of $0$ in $\Bbb{C}^p$ such that sphere $\Bbb{S}^{2n-1}_{\epsilon_0}$ intersect $f^{-1}(c)$ transversally for any $c\in U$. Let $\Bbb{B}^{2n}_{\epsilon_0}$ be the closed ball of radius $\epsilon_0$ and $D(f) \subset U$ be the set of the critical values (the discriminant set) of  the restriction of $f$ to $X^{*}:=f ^{-1}(U) \cap \Bbb{B}^{2n}_{\epsilon_0}$. By the  fibration theorem of Ehreshmann, the restriction 
$$f|_{X^{*}\setminus f^{-1}(D(f))}: X^{*}\setminus f^{-1}(D(f))\longrightarrow U \setminus D(f)$$
 is a locally $C^{\infty}$-trivial fibration. This fibration is called the Milnor fibration at the origin  and its fiber $F_0(f)= f^{-1}(t)\cap \Bbb{B}^{2n}_{\epsilon_0}$ is called the Milnor fiber of $f$ at the origin, where $t\in U \setminus D(f)$ (see \cite{Loo1984}). By a
result of Hamm \cite{Hamm1971}, the Milnor fiber $F_0(f)$ is a non-singular analytic manifold
which is homotopically equivalent to a bouquet of real spheres of dimension $n-p$. The number
of such spheres is called the Milnor number of $f$ at $0$ and is denoted by $\mu_0(f)$ (see \cite{Milnor1968} for the case $p=1$). For each closed submanifold $C\subset U\setminus D(f)$,  the Milnor fibration generates a fibration:
$$f|_{f^{-1}(C)\cap \Bbb{B}^{2n}_{\epsilon_0}}: f^{-1}(C)\cap \Bbb{B}^{2n}_{\epsilon_0}\longrightarrow C.$$
 We call this the Milnor fibration over $C$, or  the monodromy over $C$ of $f$.
 
In this paper we are  interested in the Milnor fibration and the Milnor number of (Khovanskii) non-degenerate (in the sense of \cite{Oka1990, Oka1997}) complete intersection with isolated singularity at the origin. We show that if $f,g\colon (\mathbb{C}^n,0) \rightarrow (\mathbb{C}^p,0)$ are two germs of analytic mappings with the same  Newton polyhedra  which are non-degenerate and their zero sets  are  complete intersections with isolated singularity at the origin then there is a piecewise analytic family $\{f_t\}$ of analytic maps with $f_0=f, f_1=g$ which has an {\it uniform stable radius for the Milnor fibration} and for each $t$, $f_t^{-1}(0)$ is a germ of a complete intersection.  This is  a generalization of \cite[Theorem~2.1]{Oka1979} where the case $p=1$ was considered. We refer to \cite{Oka1973,Eyral2017,Eyral-Oka2017} for other studies involved uniform stable radius in family of hypersurfaces and to \cite[Chapter 5]{Oka1997} for the one in family of complete intersections.  As a consequence, we prove that the Milnor numbers of $f$ and $g$ at the origin are equal; also, in case $n-p\neq 2$, it implies that the zero sets $f^{-1}(0), g^{-1}(0)$ are topological equivalent. That gives an extension for a result in the case $p=1$ by J. Briancon in an unpublished lecture note (see \cite{Briancon}). Finally, we give a formula for the Milnor number of  a non-degenerate isolated complete intersection singularity which is described in terms of the Newton polyhedra of the component functions. This result initiates from the work of A. G. Kouchnirenko in  the case $p=1$ (\cite{Kouchnirenko1976}) and C. Bivia-Ausina in the case $p>1$ with convenient setting and stronger notion of  non-degeneracy (\cite{Bivia2007}).
Our results also give an analog of the $\mu$-constance Theorem due to Le Dung Trang and C. P. Ramanujam (\cite{Le1976}). Similar observations for global settings were considered in \cite{HaHV1996}, \cite{HaHV1997}, \cite{Phamts2008}, \cite{Phamts2010}, \cite{Thang2019}.

\section{Preliminaries} \label{Preliminary}

In this section we present some notations and definitions, which are used throughout this paper.

\subsection{Notations}

We suppose $1 \leqslant  n \in \mathbb{N}$ and abbreviate $(x_1, \ldots, x_n)$ by $x.$   The inner product (resp., norm) on $\mathbb{C}^n$  is denoted by $\langle x, y \rangle$ for any $x, y \in \mathbb{C}^n$ (resp., $\| x \| := \sqrt{\langle x, x \rangle}$ for any $x \in \mathbb{C}^n$). The  complex conjugate of a complex number $c \in \mathbb{C}$ are denoted by  $\overline{c}$.

For each $\epsilon> 0,$ we will write $\Bbb{B}^{2n}_{\epsilon} := \{x \in \mathbb{C}^n \ : \ \|x\| \leq \epsilon\}$ for the closed ball and write $\mathbb{S}^{2n - 1}_{\epsilon} :=
\{x \in \mathbb{C}^n \ : \ \|x\| = \epsilon\}$ for the sphere.

Given nonempty sets $I \subset \{1, \ldots, n\}$ and $A =\mathbb{C}$ or $A =\mathbb{R},$ we define
$$A^I := \{x \in A^n \ : \ x_i = 0 \textrm{ for all } i \not \in I\}.$$

Let $\mathbb{C}^* := \mathbb{C} \setminus \{0\}$ and we denote by $\mathbb{Z}_+$ the set of non-negative integer numbers. If $\alpha = (\alpha_1, \ldots, \alpha_n) \in \mathbb{Z}_+^n,$ we denote by $x^\alpha$ the monomial $x_1^{\alpha_1} \cdots x_n^{\alpha_n}.$

The gradient of an analytic function defined in a neighbourhood of the origin $h \colon (\mathbb{C}^n, 0) \to (\mathbb{C}, 0)$  is denoted by $\nabla h$ as usual, i.e.,
$$\nabla h(x) := \left(\overline{\frac{\partial h}{\partial x_1}(x), \ldots, \frac{\partial h}{\partial x_n}(x)} \right),$$
so the chain rule may expressed by the inner product $\partial h/\partial {\mathbf{v}} = \langle \mathbf{v}, \nabla h \rangle.$

\subsection{Newton polyhedra and non-degeneracy conditions}

Let $h \colon (\mathbb{C}^n, 0) \to (\mathbb{C}, 0)$ be an analytic  function germ at the origin such that $h(0)=0$. Suppose that $h$ is written as $h = \sum_{\alpha} a_\alpha x^\alpha.$ Then
the support of $h,$ denoted by $\mathrm{supp}(h),$ is defined as the set of those $\alpha  \in \mathbb{Z}_+^n$ such that $a_\alpha \ne 0.$
The {\em Newton polyhedron}  of $h$, denoted by $\Gamma_{+}(h),$ is defined as the convex hull in $\mathbb{R}^n$ of the union of $\{\alpha+ \Bbb{R}^n_{+}\}$ for  $\alpha\in \mathrm{supp}(h).$ The {\em Newton boundary}  of $h$, denoted by $\Gamma(h),$ is by definition the union of compact boundary of $\Gamma_{+}(h).$  We say that a subset $\Gamma_{+}$ of $\Bbb{R}^n_{+}$ is a {\it Newton polyhedron} if there is a subset $A\subset \Bbb{Z}^n_{+}$ such that $\Gamma_{+}$ is equal to the convex hull of the set $\{\alpha + v \ : \ \alpha\in A, v\in \Bbb{R}^n_{+}\}$. A Newton polyhedron $\Gamma_{+}$ is said to be {\it convenient} if it intersects each coordinate axis in a point different from the origin.
The function $h$ or its Newton boundary is said to be {\em convenient} if $\Gamma(h)$ is convenient.
For each (compact) face $\Delta$ of $\Gamma_{+}(h),$  we will denote by $h_\Delta$ the polynomial $\sum_{\alpha \in \Delta} a_\alpha x^\alpha;$ if $\Delta \cap \mathrm{supp}(h) = \emptyset$ we let $h_\Delta := 0.$

Given a nonzero vector $q \in \mathbb{R}_{\geq 0}^n,$ we define
\begin{eqnarray*}
d(q, \Gamma_{+}(h)) &:=& \min \{\langle q, \alpha \rangle \ : \ \alpha \in \Gamma_{+}(h)\}, \\
\Delta(q, \Gamma_{+}(h)) &:=& \{\alpha \in \Gamma_{+}(h) \ : \ \langle q, \alpha \rangle = d(q, \Gamma_{+}(h)) \}.
\end{eqnarray*}
It is easy to check that for each nonzero vector $q \in \mathbb{R}_{\geq 0}^n,$ $\Delta(q, \Gamma_{+}(h)) $ is a closed face of $\Gamma_{+}(h).$ Conversely, if $\Delta$ is a closed face of $\Gamma_{+}(h)$ then there exists a nonzero vector $q \in \mathbb{R}_{\geq 0}^n$ such that $\Delta = \Delta(q, \Gamma_{+}(h)).$ 

\begin{remark}{\rm
The following statements follow immediately from definitions:

(i) For each nonempty subset $I$ of $\{1, \ldots, n\},$ if the restriction of $h$ on $\mathbb{C}^I$ is not identically zero, then $\Gamma_{+}(h) \cap \mathbb{R}^I = \Gamma_{+}(h|_{\mathbb{C}^I}).$ Also, for every nonzero vector $q=(q_1, \ldots, q_n) \in \mathbb{R}^I$ with $q_i>0$ if $i\in I$ and $\Delta:= \Delta(q, \Gamma_{+}(h|_{\Bbb{C}^I})),$ one can find a strictly positive vector $q^{'} \in \mathbb{R}_{> 0}^n$ such that $\Delta= \Delta(q^{'}, \Gamma_{+}(h)).$

(ii) Let $\Delta := \Delta(q, \Gamma_{+}(h))$ for some nonzero vector $q := (q_1, \ldots, q_n) \in \mathbb{R}_{\geq 0}^n.$ By definition, $h_\Delta = \sum_{\alpha \in \Delta} a_\alpha x^\alpha$ is a weighted homogeneous polynomial of type $(q, d := d(q, \Gamma_{+}(h))),$ i.e., we have for all $t $ and all $x \in \mathbb{C}^n,$
$$h_\Delta(t^{q_1} x_1,  \ldots, t^{q_n} x_n) = t^d h_\Delta(x_1, \ldots, x_n).$$
This implies the Euler relation
\begin{equation*}
\sum_{i = 1}^n q_i x_i \frac{\partial h_\Delta}{\partial x_i}(x) = d h_\Delta(x).
\end{equation*}
}\end{remark}

Now, let $f=(f^1, \ldots, f^p) \colon (\mathbb{C}^n, 0) \rightarrow (\mathbb{C}^p, 0)$ be an analytic mapping germ at the origin in $\Bbb{C}^n$ such that $f(0)=0$. By $\Gamma(f)$ ($\Gamma_{+}(f)$, correspondingly) we mean the $p-$tupe $(\Gamma(f^1), \Gamma(f^2), \ldots, \Gamma(f^p))$ ($(\Gamma_{+}(f^1), \Gamma_{+}(f^2), \ldots, \Gamma_{+}(f^p))$ ) and we call it the Newton boundary (Newton poyhedron) of the map $f$. We say $f$ is convenient if all coordinate functions $f^j, j=1, \ldots, p$ are convenient. The following definition of non-degeneracy is inspired from the work of Kouchnirenko~\cite{Kouchnirenko1976}, where the case $S = \mathbb{C}^n$ was considered (see also \cite{Oka1990}, \cite{Oka1997}).

\begin{definition}\label{Definition21}{\rm
We say that $f$ is {\em (Khovanskii) non-degenerate} if, for any strictly positive weight vector $q \in \mathbb{R}_{> 0}^n$  we have
\begin{eqnarray*}
f_{\Delta}^{-1}(0)\cap \{x\in \Bbb{C}^n\ : \ \rank (Df_{\Delta}(x))<p\}\subset \{x \in \mathbb{C}^{n} \ : \ x_1\ldots x_n=0\};
\end{eqnarray*}
where $f_{\Delta}$ denotes the map $(f^1_{\Delta_1}, \cdots, f^p_{\Delta_p}) $ with   $\Delta_j := \Delta(q, \Gamma_{+}(f^j))$ for $j=1, \ldots, p.$
}\end{definition}

Now we recall another notion of non-degeneracy introduced by Bivia-Ausina in \cite{Bivia2007}.

Let $\mathcal{O}_n:=\mathcal{O}_{\Bbb{C}^n, 0}$ be the ring of germs of analytic functions at $0\in \Bbb{C}^n$. Consider several germs of analytic functions $g^1,\dots,g^p\in \mathcal{O}_n$, for $p\leq n$. Take
Minkowski sum of
their Newton polyhedra, $\Gamma_{+}:=\Gamma_{+}^1+\cdots+\Gamma_{+}^p, \Gamma_{+}^i= \Gamma_{+}(g^i)$.
Let $\sigma$ be a compact face of $\Gamma_{+}$. By \cite[Lemma 3.4]{Bivia2007} there exists the unique
set  of compact faces, $\sigma_1\subset\Gamma(g^1)$, \dots, $\sigma_p\subset\Gamma(g^p)$ satisfying:
 $\sigma=\sigma_1+\cdots+\sigma_p$.
The part of $g^i$ supported on $\sigma_i$ will also be  denoted by $(g^i)_{\sigma_i}$.

\begin{definition}{\rm (See \cite[Definition 3.5]{Bivia2007})
We say that the sequence $g^1,\dots,g^p$ satisfies the
{\it $(B_\sigma)$ condition} if $\left\{x\in \Bbb{C}^n\ : \ (g^1)_{\sigma_1}(x)=\cdots=(g^p)_{\sigma_p}(x)=0\right\}\cap(\Bbb{C}^*)^n=\emptyset$.

Let $J$ be the ideal of  $\mathcal{O}_n$ generated by $g^1,\dots,g^p$. We say that the sequence $g^1,\dots,g^p$ is a {\it non-degenerate sequence} if: the ring $\mathcal{O}_n/J$ has dimension $n-p$ and the sequence satisfies the $(B_\sigma)$ condition for all the
 compact faces $\sigma$ of $\Gamma_{+}$ of dimension $\dim(\sigma)\leq p-1$.
}
\end{definition}

To define the non-degeneracy of the map $f=(f^1,\dots,f^p)$ we need the notion of
non-degeneracy of modules.

For any ideal $J$ the Newton polyhedron of $J$ is defined by
 $$\Gamma_{+}(J)=\Conv\Big({\mathop\cup\limits}_{f\in J}(\Supp(f)+\Bbb{R}^n_{\geq 0})\Big).$$

Consider a submodule of a free module, $M\subset \mathcal{O}_n^{\oplus p}$. Denote by $A_M$
its generating
 matrix, i.e. a $p\times s$ matrix with entries in $\mathcal{O}_n$, whose columns generate
 the module.
 Denote by $M_i$ the ideal in $\mathcal{O}_n$ generated by the entries of $i$'th row of
 $A_M$.
  (It does not depend on the choice of generators of the module.) The {\em Newton polyhedron of
  $M$} is
   defined to be $\Gamma_{+}(M):=\Gamma_{+}(M_1)+\cdots+\Gamma_{+}(M_p)$.
   (Here each $M_i$ is an ideal and we use the definition of $\Gamma_{+}(J)$ as above.
   In the case of one-row-matrix $M$ itself is an ideal.)
    For any compact face $\sigma$ of $\Gamma_{+}(M)$ take
   its (unique) presentation $\sigma=\sigma_1+\cdots+\sigma_p$, $\sigma_i\subset\Gamma_{+}(M_i)$, as above.
   Denote by $M|_\sigma$ the matrix of restrictions, its $i$'th row consists of the restrictions
   onto $\sigma_i$.
   (Note that all the restrictions are polynomials, not just power series.)

\begin{definition}{\rm
The module/matrix $M$ is called {\it Newton-non-degenerate} if for any compact face $\sigma\subset\Gamma_{+}(M)$ the following property holds:
\[\{x\in\Bbb{C}^n:\ \rank(M|_\sigma(x))\leq p\}\cap(\Bbb{C}^{*})^n=\emptyset.\]

}
\end{definition}

Finally, for a map $f=(f^1,\dots,f^p):(\Bbb{C}^n,0)\to(\Bbb{C}^p,0)$ consider a version of degeneracy
matrix,
 describing the singular locus:
\begin{equation*}
N(f):=\begin{pmatrix} x_1\frac{\partial f^1}{\partial x_1}&\dots&x_n\frac{\partial f^1}{\partial x_n}&f^1&\dots&0\\
\dots&\dots&\dots&\dots&\dots&\dots\\
x_1\frac{\partial f^p}{\partial x_1}&\dots&x_n\frac{\partial f^p}{\partial x_n}&0&\dots&f^p
\end{pmatrix}.
\end{equation*}

\begin{definition}\label{newtonnondegeneratedef}{\rm (See \cite[Definition 3.5]{Bivia2007})
Consider the convenient map $f=(f^1,\dots,f^p):(\Bbb{C}^n,0)\to(\Bbb{C}^p,0)$.
 The map $f$ is called {\it Newton-non-degenerate} if
\\(i) the sequences $f^1,\dots,f^r$ are non-degenerate for any $r=1,\dots,p-1$, and
\\(ii) the submodule $N(f)\subset\mathcal{O}_n^{\oplus p}$ is Newton-non-degenerate.
}
\end{definition}

\begin{remark}
{\rm The notion of Newton-non-degenerate in the sense of Definition \ref{newtonnondegeneratedef} is stronger than the notion of Khovanskii non-degeneracy in the sense of Definition \ref{Definition21}, see \cite[Example 6.10]{Bivia2007} for detail.
 }
\end{remark}

\subsection{Mixed Newton numbers}

 Let $\Gamma_{+}\subset \Bbb{R}^n_{\geq 0}$ be a  convenient Newton polyhedron polyhedron, its covolume is defined as $\covol(\Gamma_{+}):= \vol_n(\Bbb{R}^n\setminus \Gamma_{+})$, where $\vol_n$ denotes the normalized $n$-dimensional volume in $\Bbb{R}^n$. For a collection $\Gamma_{+}^1, \ldots, \Gamma_{+}^p$ of convenient polyhedra, we consider the scaled Minkowski sum $\lambda_1\Gamma_{+}^1+\cdots \lambda_p\Gamma_{+}^p$. Its covolume is a polynomial in $\lambda_i$ (by \cite[Theorem 10.4]{KK}):
$$\covol(\lambda_1\Gamma_{+}^1+\cdots \lambda_p\Gamma_{+}^p)= \sum \left(\begin{array}{ccc}
n \\
k_1, \ldots, k_p
\end{array}\right)\covol\Big((\Gamma_{+}^1)^{k_1}, \ldots, (\Gamma_{+}^p)^{k_p}\Big)\left(\prod_{i=1}^p\lambda_i^{k_i}\right)$$
where the sum runs over all tupes $k_1, \ldots, k_p$ which $k_i\geq 0, k_1+\ldots + k_p=n$, and 
$$ \left(\begin{array}{ccc}
n \\
k_1, \ldots, k_p
\end{array}\right):= \frac{n!}{k_1!\cdots k_p!}.$$
 The mixed covolumes are the coefficients $\covol\Big((\Gamma_{+}^1)^{k_1}, \ldots, (\Gamma_{+}^p)^{k_p}\Big)$ from the above equation. Here  $\covol\Big((\Gamma_{+}^1)^{k_1}, \ldots, (\Gamma_{+}^p)^{k_p}\Big)$ is a shortand for $\covol\Big(\underbrace{\Gamma_{+}^1,\dots,\Gamma_{+}^1}_{k_1},\dots,\underbrace{\Gamma_{+}^p,\dots,\Gamma_{+}^p}_{k_p}\Big)$. See  \cite[Section 2.3]{KN} for further properties of covolumes.

\begin{definition}\label{definitionMixednumber}{\rm (See \cite[Definition 3.2]{Bivia2007} and \cite[Section 2.6.1]{KN}) Let $\Gamma_{+}^1, \ldots, \Gamma_{+}^p$ be convenient polyhedra in $\Bbb{R}^n_{+}$. The {\it mixed Newton number} of $\Gamma_{+}^1, \ldots, \Gamma_{+}^p$ is defined as:
$$\nu(\Gamma_{+}^1, \ldots, \Gamma_{+}^p)= \sum\limits^{n}_{j=p}(-1)^{n-j}\Big(\sum\limits_{\substack{I\subseteq\{1,\dots,n\}\\|I|=j}}j!
a_j((\Gamma_{+}^1)^I,\dots,(\Gamma_{+}^p)^I)\Big)+(-1)^{n-p+1},$$
 where $$a_j((\Gamma_{+}^1)^I,\dots,(\Gamma_{+}^p)^I):=\sum\limits_{\substack{k_1+\cdots + k_p=j\\k_1,\dots,k_p\geq 1}}
\covol_j\Big(((\Gamma_{+}^1)^I)^{k_1},\dots,((\Gamma_{+}^p)^I)^{k_p}\Big)$$

and $(\Gamma_{+}^j)^I=\Gamma_{+}^j\cap \Bbb{R}^I.$ 
 The coefficient $\covol_j\Big(((\Gamma_{+}^1)^I)^{k_1},\dots,((\Gamma_{+}^p)^I)^{k_p}\Big)$ is the
 $j$-dimensional   mixed-covolume defined above.

}
\end{definition}

\begin{remark}
{\rm  For a Newton polyhedron $\Gamma_{+}$ in $\Bbb{R}^n_{+}$, we have $\covol\Big(\underbrace{\Gamma_{+},\dots,\Gamma_{+}^1}_{n}\Big)=\vol (\Bbb{R}^n_{+}\setminus ~\Gamma_{+})$. Then, if $p=1$ we have $\nu(\Gamma_{+})= \sum_{i=0}^n(-1)^{n-i}i!\vol_i(\Bbb{R}^n_{+}\setminus \Gamma_{+})$ which is the Newton number defined by Kouchnirenko in \cite{Kouchnirenko1976}.

}
\end{remark}

\section{Uniform stable radius} \label{Section3}

Let $F(t,x)=(F^1(t,x), \ldots, F^p(t,x))\colon [0, 1]\times \Bbb{C}^n\to \Bbb{C}^p$ be a mapping such that $F$ is real analytic on $t$ and for each $t\in [0, 1]$ the map $f_t(x):= F(t, x)$ is analytic in some neighbourhood of the origin in $\Bbb{C}^n$ with $f_t(0)=0.$ For each $t\in [0, 1]$ and each $j=1, \ldots, p$, we denote by $f^j_t$ the function $x\mapsto F^j(t, x)$. 

\begin{definition}{\rm (See \cite{Oka1979})
 We say that the positive number $\epsilon_0>0$ is a {\it uniform stable radius} for the Milnor fibration of the family $\{f_t\}_{ t\in [0, 1]}$ if  for each $\epsilon \leqslant \epsilon_0$ and each $t\in [0, 1],$ the set $f_t^{-1}(0)$ intersects transversally with the sphere $\mathbb{S}^{2n-1}_{\epsilon}$.

}
\end{definition}

The following fact is presented in  \cite{Oka1997} as conclusion $(T_{\epsilon})$, though we provide here a different and simpler proof.
 
\begin{lemma}(\cite{Oka1997})\label{Lemma31} 
Suppose that the family $\{f_t\}_{t\in [0, 1]}$ satisfies:
\begin{itemize}
\item[(i)] For each $j=1, \ldots, p$, the Newton boundary of $f^j_t$ is convenient and does not depend on $t\in [0, 1]$;
\item[ii)] For each $t\in [0, 1]$,  the map $f_t$ is 
 non-degenerate.
\end{itemize}
Then, the family  $\{f_t\}_{ t\in [0, 1]}$ has a uniform stable radius.
\end{lemma}

\begin{proof}
Assume by contradiction that the family does not have any uniform stable radius, then there exist sequences $\{t^k\}_{k\in \mathbb{N}}\subset [0, 1], \{\epsilon^k\}_{k\in \mathbb{N}}\to 0$ such that the sets  $f_{t^k}^{-1}(0)$ do not intersect transversally with the spheres $\mathbb{S}^{2n-1}_{\epsilon^k}$. Therefore, there exist sequences $\{x^k\}_{k\in \mathbb{N}} \subset \mathbb{C}^n$ and $\{\lambda_j^k\}_{k\in \mathbb{N}} \subset \mathbb{C}, j=1,\ldots, p+1, $ such that
\begin{enumerate}
\item[(a1)] $\Vert x^k \Vert \rightarrow 0, \Vert x^k \Vert \neq 0$ as $k\rightarrow \infty;$
\item[(a2)] $F(t^k, x^k)=f_{t^k}(x^k) = 0$ for all $k \in \mathbb{N};$
\item[(a3)] $\sum_{j=1}^p\lambda_j^k\nabla f^j_{t^k}(x^k) = \lambda_{p+1}^k {x^k};$
\item[(a4)] The numbers $\lambda_j^k, j = 1, \ldots, p+1,$ are not all zero for any $k\in \mathbb{N}.$
\end{enumerate}
By the Curve Selection Lemma (see \cite{Milnor1968}), there exist analytic curves
\begin{eqnarray*}
\phi=(\phi_1, \ldots, \phi_n) \colon [0,\epsilon) \rightarrow \mathbb{C}^n, \quad  t \colon [0,\epsilon) \rightarrow [0, 1] \quad \textrm{ and } \quad  \lambda_j \colon (0,\epsilon) \rightarrow \mathbb{C}, \ j = 1,\ldots,p+1,
\end{eqnarray*}
such that
\begin{enumerate}
\item[(a5)] $\Vert \phi(s) \Vert \rightarrow 0,$ as $ s\rightarrow 0$ and $\phi(s) \neq 0$  for all $s \in (0,\epsilon);$
\item[(a6)] $F(t(s),\phi(s)) = 0$ for all $s \in (0,\epsilon);$
\item[(a7)] $\sum_{j=1}^p\lambda_j(s)\nabla f^j_{t(s)}(\phi(s)) = \lambda_{p+1}(s) {\phi(s)}\textrm{ for all } s\in (0,\epsilon);$
\item[(a8)] $\lambda_j(s), j = 1, \ldots, p + 1,$ are not all zero for any $s\in(0,\epsilon).$
\end{enumerate}

Put $I:=\{ i \ : \ \phi_i \not\equiv 0 \}$. By the condition (a5), $I \neq \emptyset$. For $i \in I$, we can write the curve $\phi_i$ in terms of parameter as follows
$$\phi_i(s) \ =\ x^0_i s^{q_i} + \textrm{higher-order terms in }s,$$
where $x_i^0 \neq 0$, and $q_i \in \mathbb{Q}.$ We have $\min_{i \in I} q_i > 0,$ due to the condition (a5). We also write $t(s)$ as
$$t(s) \ =t^0+t^1s^{q} + \textrm{higher-order terms in }s,$$
where $t^0=\lim_{s\to 0} t(s)\in [0, 1], t^1\in \Bbb{R}$ and $q>0$.

For each $j=1, \ldots, q$ and each $ t\in [0, 1]$, $f^j_t(x)$ is convenient then $f^j_t|_{\mathbb{C}^I}\not\equiv 0.$  Let $d_j>0$ be the minimal value of the linear function $\sum_{i \in I}\alpha_i q_i$ on $ \mathbb{R}^I \cap \Gamma_{+}(f^j_t)$ and $\Delta_j$ be the maximal face of $\mathbb{R}^I \cap \Gamma_{+}(f^j_t)$, where this linear function attains its minimum value. Remark that the Newton polyhedrons $\Gamma_{+}(f^j_t)$ do not depend on $t$. It is easy to check that
\begin{eqnarray*}
F^j(t(s),\phi(s)) &=& (f^j_{t^0})_{\Delta_j}(x^0) s^{d_j} + \textrm{ higher-order terms in }s,
\end{eqnarray*}
where $x^0 := (x^0_1, \ldots, x^0_n)$ with $x^0_i = 1$ for $i \not \in I$ and $(f^j_{t^0})_{\Delta_j}$ is the face function associated with $f^j_{t^0}$ and $\Delta_j$ which does not depend on the variables $x_i$ if $i\not\in I$. It implies from the condition (a6) that
\begin{eqnarray} \label{EQ03.1}
(f^j_{t^0})_{\Delta_j}(x^0)=0, \quad \textrm{for all}\quad j=1, \ldots, p.
\end{eqnarray}

For $i \in I$ and $j \in \{1, \ldots, p\}$, we also have:
\begin{eqnarray*}
\frac{\partial F^j(t(s),\phi(s))}{\partial x_i}  &=& \frac{\partial (f^j_{t^0})_{\Delta_j}}{\partial x_i} (x^0) s^{d_j - q_i} + \textrm{ higher-order terms in }s.
\end{eqnarray*}
It follows from (a7) and (a8) that one of the functions $\lambda_1(s), \ldots, \lambda_p(s)$ is not equal to zero. For $j \in \{1, \ldots, p\}$ which $\lambda_j(s)\not\equiv 0$ we write
$$\lambda_j(s)=c_js^{\beta_j} + \textrm{ higher-order terms in }s,\quad c_j\neq 0.$$
 Put 
$$e:= \min \left\{\beta_l+d_l\ : \ l\in\{1, \ldots, p\}\quad \textrm{which}\quad  \lambda_l(s)\not\equiv 0\right\}$$
and
$$J:=\{j \ :\ \lambda_j(s)\not\equiv 0, \beta_j+d_j=e\}.$$
Then the condition (a7) is equivalent to the following:
\begin{eqnarray}\label{EQ03.3}
\left( \sum_{j \in J} \overline{c_j} \frac{\partial (f^j_{t^0})_{\Delta_j}}{\partial x_i}(x^0)\right) s^{e-q_i}  + \cdots &  = & \overline{\lambda_{p+1}(s)\phi_i(s)}\quad \textrm{ for all}\quad  i\in I,
\end{eqnarray}
where dots stand for  higher-order terms in $s$.

If $\lambda_{p+1}(s)\equiv 0:$  for all $i\in I$ we get
$$ \sum_{j \in J} \overline{c_j} \frac{\partial (f^j_{t^0})_{\Delta_j}}{\partial x_i}(x^0)=0. $$
Hence
\begin{eqnarray}\label{EQ03.2}
 \sum_{j \in J} \overline{c_j} \nabla (f^j_{t^0})_{\Delta_j}(x^0)=0. 
\end{eqnarray}
Two equalities (\ref{EQ03.1}) and (\ref{EQ03.2}) imply the contradiction to the nondegeneracy of $f_{t^0}.$

If $\lambda_{p+1}(s)\not\equiv 0:$ we also write $\lambda_{p+1}(s)$ as
$$\lambda_{p+1}(s)=c_{p+1}s^{\beta_{p+1}} + \textrm{ higher-order terms in }s,\quad c_{p+1}\neq 0.$$
The equation (\ref{EQ03.3}) becomes
\begin{eqnarray}\label{EQ03.4}
\left( \sum_{j \in J} \overline{c_j} \frac{\partial (f^j_{t^0})_{\Delta_j}}{\partial x_i}(x^0)\right) s^{e-q_i}  + \cdots &  = & \overline{c_{p+1} x^0_i}s^{\beta_{p+1}+q_i} + \cdots\quad \textrm{for all}\quad  i\in I.
\end{eqnarray}
 Since $c_{p+1}$ and $x^0_i$ are nonzero for all $i\in I$, we get 
$e-q_i\leq \beta_{p+1}+q_i$ for all $i\in I$. If $e-\beta_{p+1}< 2\min_{i\in I} q_i$ then by the same argument as above we obtain a contradiction to the non-degeneracy condition of $f_{t^0}$. 
Otherwise, if $e-\beta_{p+1}= 2\min_{i\in I} q_i> 0$. Denote 
$$I_1:=\{i\in I\ :\ q_i=  \min_{l\in I} q_l\}.$$
The equation (\ref{EQ03.4}) gives us the following
\begin{eqnarray*}
\sum_{j \in J} \overline{c_j} \frac{\partial (f^j_{t^0})_{\Delta_j}}{\partial x_i}(x^0) & = &
\begin{cases}
\overline{c_{p+1}x^0_i} & \textrm{ if } i \in I_1, \\
0  & \textrm{ if } i \in I \setminus I_1, \\
0  & \textrm{ if } i \not \in I,
\end{cases}
\end{eqnarray*}
the last equation holds because for all $i \not \in I$ and all $j \in J,$ the polynomial $(f^j_{t^0})_{ \Delta_j}$ does not depend on the variable $x_i.$
Consequently,
\begin{eqnarray*}
\sum_{i = 1}^n \left(\sum_{j \in J} \overline{c_j} \frac{\partial (f^j_{t^0})_{\Delta_j}}{\partial x_i}(x^0)\right) x^0_i q_i & = & \nonumber
\sum_{i \in I_1} \left(\sum_{j \in J} \overline{c_j} \frac{\partial (f^j_{t^0})_{\Delta_j}}{\partial x_i} (x^0)\right) x^0_i q_i \\
& = & \sum_{i\in I_1} \overline{c_{p+1}}|x^0_i|^2 \frac{e-\beta_{p+1}}{2}.
\end{eqnarray*}

On the other hand, by the Euler relation, for all $j \in J,$ we have 
\begin{eqnarray*}
\sum_{i = 1}^n  \frac{(f^j_{t^0})_{\Delta_j}}{\partial x_i}(x^0) x^0_i q_i & = &  d_j (f^j_{t^0})_{\Delta_j}(x^0).
\end{eqnarray*}
Combining this equality and the equation (\ref{EQ03.1}) we get 
\begin{eqnarray*}
\sum_{i = 1}^n \left(\sum_{j \in J} \overline{c_j} \frac{\partial (f^j_{t^0})_{\Delta_j}}{\partial x_i}(x^0)\right) x^0_i q_i
& = &  \sum_{j \in J} \overline{c_j} \left( \sum_{i = 1}^n  \frac{\partial (f^j_{t^0})_{\Delta_j}}{\partial x_i}(x^0) x^0_i q_i \right) \\
& = &  \sum_{j \in J} \overline{c_j} d_j (f^j_{t^0})_{\Delta_j}(x^0)\\
& = & 0.
\end{eqnarray*}
Therefore $\sum_{i\in I_1} \overline{c_{p+1}} |x^0_i|^2 \frac{e-m_{p+1}}{2} =  0.$ 
This is a contradiction.
\end{proof}

\begin{remark}
{\rm It implies from the proof of Lemma \ref{Lemma31} that for each $t$ the zero set $f_t^{-1}(0)$ has at most isolated singularity at the origin.

}
\end{remark}

Now, we work with the non-convenient case. Let $f(x)=(f^1, \ldots, f^p)(x): (\Bbb{C}^n, 0)\to (\Bbb{C}^p, 0)$ be an analytic mapping germ such that $V:=f^{-1}(0)$ is a germ of a complete intersection with an isolated singularity at the origin. 
 Let $\mathbf{m}$ be the maximal ideal of $\mathcal{O}_{n}$. Let $J_f$ be the ideal of $\mathcal{O}_{n}$ generated by $f^1, \ldots, f^p$ and determinants of maximal order minors of the Jacobian matrix of $f$. Since $f^{-1}(0)$ has isolated singularity at the origin, by the Hilbert nullstellensatz (\cite[Proposition 1.1.29]{Huybr}), we have
$$\mathbf{m}\subset \sqrt{J_f}.$$
Let $\mu\in \Bbb{N}$ be the smallest number such that
$$x_i^{\mu}\in J_f, \quad \forall i=1, \ldots, n.$$

\begin{lemma}\label{lem3.3}
	With the above notation and assumption,  consider the family 
	$$f_t(x):=F(t,x):=(f^1(x)+tx^{\nu}, f^2(x), \ldots, f^p(x)), \quad t\in [0, 1],$$
    where $\nu=(\nu_1, \ldots, \nu_n)\in \Bbb{N}^n$ satisfying $|\nu|= \sum_{i=1}^n\nu_i \geq \mu+2$ and for each $i$, either $\nu_i=0$ or $\nu_i\geq 2$. Then:
\begin{itemize}
 \item[1)] The family $\{f_t\}_{t\in [0, 1]}$ has an uniform stable radius for the Milnor fibration.
 \item[2)] For some $\nu$ large enough, for each $t$, the zero set $f_t^{-1}(0)$ is a complete intersection; and, if $f$ is non-degenerate, for generic $t$, the map $f_t$ is non-degenerate.
 \end{itemize}
\end{lemma}

\begin{proof}
1) Suppose that such uniform stable radius does not exist. Then, by the same argument as in the proof of Lemma \ref{Lemma31}, we can find real analytic functions:
\begin{eqnarray*}
\phi=(\phi_1, \ldots, \phi_n) \colon [0,\epsilon) \rightarrow \mathbb{C}^n, \quad  t \colon [0,\epsilon) \rightarrow [0, 1] \quad \textrm{ and } \quad  \lambda_j \colon (0,\epsilon) \rightarrow \mathbb{C}, \ j = 1,\ldots,p+1,
\end{eqnarray*}
such that
\begin{enumerate}
\item[(1)] $\Vert \phi(s) \Vert \rightarrow 0,$ as $ s\rightarrow 0$ and $\phi(s) \neq 0$  for all $s \in (0,\epsilon);$
\item[(2)] $F(t(s),\phi(s)) = 0$ for all $s \in (0,\epsilon);$
\item[(3)] $t(s)\lambda_1(s)(\nabla x^{\nu})(\phi(s))+\sum_{j=1}^p\lambda_j(s)(\nabla f^j)(\phi(s)) = \lambda_{p+1}(s) {\phi(s)}\textrm{ for all } s\in (0,\epsilon);$
\item[(4)] $\lambda_j(s), j = 1, \ldots, p + 1,$ are not all zero for any $s\in(0,\epsilon).$
\end{enumerate}
We expand those functions as follows:
\begin{align*}
\phi_i(s)&=x^0_is^{q_i}+ \cdots, i=1, \ldots, n \\
t&=t^0+t^1s^q+\cdots, \\
\lambda^j&=c_js^{\beta_j}+\cdots, j=1, \ldots, p+1
\end{align*}
where $q_i>0$ for all $i$ (possibly $q_i=\infty$). For each $i$ we have $q_i=\infty$ if $\phi_i \equiv 0$, otherwise $x^0_i\in \Bbb{C}^{*}$. We also see that $t^0\in [0, 1]$ and $q>0$. Put 
$$a:= \min_{i=1, \ldots, n}\{q_i\}>0.$$
Without loss of generality, we may assume that $a=q_1$.

Denote by $F^j, j=1, \ldots, p$ the component functions of $F$. Take the derivative both sides of the condition (2), we obtain that:
$$\frac{\partial F^j}{\partial t}.\frac{dt}{ds}+ \left\langle \frac{d\phi}{ds}, \nabla F^j(\phi) \right\rangle=0, \quad\textrm{for all}\quad j=1,\ldots, p. $$
Combining these equations with the condition (3),  we get:
\begin{equation}\label{equa1}
\overline{\lambda_1}\phi^{\nu}\frac{dt}{ds}+ \overline{\lambda_{p+1}}  \left\langle \frac{d\phi}{ds}, \phi \right\rangle=0.
\end{equation}
We consider the following two possibilities:

{\bf Case 1:} $\lambda_{p+1}=0$. If $\lambda_1\neq 0$, then $\phi^{\nu}(s)=0$. Since for each $i$ either $\nu_i=0$ or $\nu_i\geq 2$, it implies from (2) and (3) that $f(\phi(s))=0$ and $\sum_{j=1, \ldots, p} \nabla f^j(\phi(s))=0$. This means $f^{-1}(0)$ has non-isolated singularities at the origin (contradiction). Otherwise, if $\lambda_1=0$ the the vectors $\nabla f^2 (\phi(s)), \ldots, \nabla f^p (\phi(s))$ are linearly dependent. Furthermore $x_1^{\mu}\in J_f$ then there exist analytic functions $g_j, h_k$ such that
$$x_1^{\mu}=\sum_{j=1,\ldots, p}g_jf^j+\sum_k h_k J_k$$
where $J_k$ are determinants of maximal order minors of the Jacobian matrix of $f$. Substitute $x=\phi(s)$ both sides of the above equation and remark that all the determinants $ J_k(\phi(s))=0$, we get
$$\phi_1^{\mu}(s)= -t(s)g_1(s)\phi^{\mu}(s).$$
  This is again a contradiction, since the order of the left hand side is $a\mu$, while the right hand side's order is not less that $a\nu.$
  
  {\bf Case 2:} $\lambda_{p+1}\neq 0$. It follows from the equation (5), by comparing the orders, that
  $$\beta_1+a\nu+q= 2a+\beta_{p+1}>\beta_1+a\nu,$$
  by the assumption $\nu>\mu+1$, we get $a+\beta_{p+1}-\beta_1>a\mu.$ 

On the other hand, it is easy to check that
$$J_f\subset J_{f_t}+ \mathbf{m}^{\nu}.$$
Hence, due to $x_1^{\mu}\in \mathbf{m}$, there exist analytic functions $g^{'}_j, h^{'}_k, p_I$ such that
$$x_1^{\mu}= \sum g_j^{'}f_t^j+ \sum h^{'}_kJ^{'}_k+\sum_{I=(i_1, \ldots, i_{\nu})\subset \{1, \ldots, n\}} p_Ix_{i_1}x_{i_2}\ldots x_{i_{\nu}}$$
where $f^j_t$ are component functions of $f_t$ vanishing along $\phi(s)$ and $J^{'}_k$ are determinants of maximal order minors of Jacobian matrix of $f_t$. Replacing $x$ by $\phi(s)$ both sides of the above equation, we get following
\begin{equation}\label{equa6}
\phi_1^{\mu}= \sum h^{'}_k(\phi(s))J^{'}_k(\phi(s))+\sum_I p_I\phi_{i_1}\phi_{i_2}\ldots \phi_{i_{\nu}}.
\end{equation}
By the condition (3), the first row of the Jacobian matrix of $f_t$ (the gradient vector of $f^1_t$) is the linear combination of the others and the vector $\frac{\lambda_{p+1}(s)}{\lambda_1(s)}\phi(s)$. Thus the order of $J^{'}_k(\phi(s))$ is not less than $a+\beta_{p+1}-\beta_1$. By comparing orders of both sides of the equation (\ref{equa6}) we get the contradiction.

2) Suppose by contradiction that for any $\nu$ large enough, there exists   $t_{\nu}\in (0, 1]$ which  $f_{t_{\nu}}^{-1}(0)$ is not a complete intersection. Then, by \cite[Section 1.6]{Loo1984} and by $f^{-1}(0)$ is a complete intersection, there are analytic functions $g_{1,\nu}, \ldots, g_{p,\nu}\in \mathcal{O}_{n}$ which $g_{1,\nu}\neq 0$, such that 
\begin{equation}\label{equa}
g_{1,\nu}(f^1+t_{\nu}x^{\nu})+g_{2,\nu}f^2+\cdots+g_{p,\nu}f^p=0.
\end{equation}
Combining that equation and the one for $2\nu$ we get that $g_{1,\nu}g_{1,2\nu}(t_{\nu}x^{\nu}-t_{2\nu}x^{2\nu})$ belongs to the ideal $\langle f^2, \ldots, f^p\rangle$. This implies $g_{1,\nu}g_{1,2\nu}x^{\nu}\in \langle f^2, \ldots, f^p\rangle$. Therefore, by multiple both sides of equation (\ref{equa}) with $g_{1,2\nu}$ we obtain that $g_{1,2\nu}g_{1,\nu}f^1$ also belongs to the ideal $\langle f^2, \ldots, f^p\rangle$. Hence, according to \cite[Section 1.6]{Loo1984} the germ $f^{-1}(0)$ is not a complete intersection. This is a contradiction.

The claim on non-degeneracy is proved by the same way as in \cite[Appendix]{Oka1979}.

\end{proof}

The main result in this section is the following.

\begin{theorem}\label{thm4.2}
Let $f,g: (\Bbb{C}^n, 0)\longrightarrow (\Bbb{C}^p, 0)$ be two germs of analytic mappings such that:
\begin{itemize}
\item[i)] $\Gamma(f)=\Gamma(g)$;
\item[ii)] $f,g$ are non-degenerate and the zero sets $f^{-1}(0), g^{-1}(0)$ are complete intersections with isolated singularity at the origin.
\end{itemize}
Then, there exists a piecewise analytic family $\{f_t\}_{t\in [0, 1]}$ of analytic maps, $f_0=f, f_1=g$ which has a uniform stable radius, and for each $t$, $f_t^{-1}(0)$ is a germ of a complete intersection.
\end{theorem}

\begin{proof} It implies from Lemma \ref{lem3.3} that there are analytic families $\{\alpha_t\}_{t\in [0, 1]}, \{\beta_t\}_{t\in [0, 1]}$ of analytic maps from $(\Bbb{C}^n, 0)$ to $ (\Bbb{C}^p, 0)$ which have  uniform stable radius $\epsilon_1, \epsilon_2$ and satisfy the followings:
\begin{itemize}
\item[1)] $\alpha_0=f, \beta_0=g, $ the Newton boundaries $\Gamma(\alpha_1)$, $\Gamma(\beta_1)$ coincide and are convenient;  
\item[2)] $\alpha_1, \beta_1$ are non-degenerate and for each $t$, $\alpha_t^{-1}(0), \beta_t^{-1}(0)$ are germs of complete intersections.
\end{itemize}
Indeed, one can choose $\{\alpha_t\}_{t\in [0, 1]}$ in the form:
 $$(f^1+t_1^1x_1^{\alpha^1_1}+\cdots + t^1_nx_n^{\alpha^1_n}, \ldots, f^p+t_1^px_1^{\alpha^p_1}+\cdots + t^p_nx_n^{\alpha^p_n})$$
for $f=(f^1, \ldots, f^p )$ and similar way for $\{\beta_t\}_{t\in [0, 1]}$.

On 	the other hand, since the non-degeneracy is an open condition (see \cite[Appendix]{Oka1979}), we can find a piecewise analytic family $\{\Phi(t, x)\}$ such that
$\Phi(0,x)=\alpha_1(x), \Phi(1, x)=\beta_1(x)$, $\Phi$ is piecewise analytic on $t\in [0, 1]$; for each $t$, the map $\phi_t(x):=\Phi(t,x)$ is analytic on some neighbourhood of the origin $0\in \Bbb{C}^n$ and the following two conditions hold: 
\begin{itemize}
\item[3)] The Newton boundaries $\Gamma(\phi_t)$ of $\phi_t$ is independent of $t$ and convenient;
\item[4)] For each $t$, the map $\phi_t$ is non-degenerate.
\end{itemize}
Then, by Lemma \ref{Lemma31} the family $\{\phi_t\}$ has a uniform stable radius $\epsilon_3$.  Also, since for each $t$, $\phi_t$ is non-degenerate and convenient, $\phi_t^{-1}(0)$ has only isolated singularity at the origin. Therefore, there exists a small ball $B\subset \Bbb{C}^n$ such that  $(\phi_t^{-1}(0)\cap B) \setminus \{0\}$ is a $(n-p)-$dimensional complex manifold which implies that $\phi_t^{-1}(0)$ is a germ of a complete intersection.

By connecting the families $\{\alpha_t\}, \{\beta_t\}$ and $\{\phi_t\}$, we get a family as desired with uniform stable radius $\epsilon:= \min\{\epsilon_1, \epsilon_2, \epsilon_3\}$.
\end{proof}

\section{Milnor number and mixed Newton number}

In this section we give some applications of the main result in previous section, more precisely, we give a formula for the Milnor number of a non-degenerate isolated complete intersection singularity in terms of the Newton polyhedrons.

Let $F(t,x)=(F^1(t,x), \ldots, F^p(t,x))\colon [0, 1]\times (\Bbb{C}^n, 0)\to (\Bbb{C}^p, 0)$ be a mapping such that $F$ is real piecewise analytic on $t$, for each $t\in [0, 1]$, the map $f_t(x):= F(t, x)$ is analytic in some neighbourhood of the origin in $\Bbb{C}^n$ and  $V_t:= f^{-1}_t(0)$ is a complete intersection with isolated singularity at the origin. For each $t$, we denote $f_t^j$ the function $x\mapsto F^j(t, x)$ and $D_t$ be the discriminant set of $f_t\colon \Bbb{C}^n\to \Bbb{C}^p$. Then $D_t\subset \Bbb{C}^p$ is a hypersurface of dimension $p-1$ (by \cite[Section 2.8]{Loo1984}). We have the following property of a family which has a uniform stable radius.

\begin{lemma}\label{Lemma3.2}
With the above notations, suppose that the family $\{f_t\}_{t\in [0, 1]}$ has a uniform stable radius $\epsilon_0$ for the Milnor fibration. Then, there exists a small neighbourhood $U$ of the origin in $\Bbb{C}^p$ such that for each $c\in U$ and each $t\in [0, 1]$, the set $f_t^{-1}(c)$ intersects transversally with the sphere $\mathbb{S}^{2n-1}_{\epsilon_0}$. 
\end{lemma}

\begin{proof}
Assume by contradiction that such neighbourhood $U$ does not exist. This means, there exist sequences $\{t^k\}_{k\in \Bbb{N}}\subset [0, 1]$ and $\{c^k\}_{k\in\Bbb{N}}\subset \Bbb{C}^p$ such that $c^k\to 0 $ and the set $f^{-1}_{t^k}(c^k)$ does not intersect the sphere $\mathbb{S}^{2n-1}_{\epsilon_0}$  transversally, for any $k\in \Bbb{N}$. Then, there  exist sequences $\{x^k\}_{k\in \mathbb{N}} \subset \mathbb{S}^{2n-1}_{\epsilon_0}$ and $\{\lambda_j^k\}_{k\in \mathbb{N}} \subset \mathbb{C}, j=1,\ldots, p+1, $ such that
\begin{enumerate}
\item[(a1)] $F(t^k, x^k)=f_{t^k}(x^k) = c^k$ for all $k \in \mathbb{N};$
\item[(a2)] $\sum_{j=1}^p\lambda_j^k\nabla f^j_{t^k}(x^k) = \lambda_{p+1}^k {x^k};$
\item[(a3)] The numbers $\lambda_j^k, j = 1, \ldots, p+1$ are not all zero for any $k\in \mathbb{N}.$
\end{enumerate}
By the Curve Selection Lemma (see \cite{Milnor1968}), there exist analytic curves
\begin{eqnarray*}
\phi \colon [0,\epsilon) \rightarrow \mathbb{S}^{2n-1}_{\epsilon_0}, \quad  t \colon [0,\epsilon) \rightarrow [0, 1] \quad \textrm{ and } \quad  \lambda_j \colon (0,\epsilon) \rightarrow \mathbb{C}, \ j = 1,\ldots,p+1,
\end{eqnarray*}
such that
\begin{enumerate}
\item[(a4)] $\phi(s) \rightarrow x^0\in \mathbb{S}^{2n-1}_{\epsilon_0}$ as $ s\rightarrow 0;$
\item[(a5)] $F(t(s),\phi(s))\rightarrow 0$ as $s \rightarrow 0;$
\item[(a6)] $\sum_{j=1}^p\lambda_j(s)\nabla f^j_{t(s)}(\phi(s)) = \lambda_{p+1}(s) {\phi(s)}\textrm{ for all } s\in (0,\epsilon);$
\item[(a7)] $\lambda_j(s), j = 1, \ldots, p + 1,$ are not all zero for any $s\in(0,\epsilon).$
\end{enumerate}

Denote 
$$J:=\{j\in\{1, \ldots, p+1\}\ :\ \lambda_j(s)\not\equiv 0\},$$
due to the condition (a7), $J\neq \emptyset$. By dividing both sides of (a6) by $s^a$, if necessary, where $a$ is the lowest order of nonzero functions $\lambda_j(s), j\in J,$  we may assume that, for all $j=1, \ldots, p+1$ there exist limits 
$$c^0_j:=\lim_{s\to 0}\lambda_j(s)$$
and the numbers $c^0_j, j = 1, \ldots, p + 1,$ are not all zero (by (a7)).
Denote $t^0:= \lim_{s\to 0}t(s)\in [0, 1]$. 

 Now, taking the limit when $s\to 0$ in the conditions (a5) and (a6) we get $f_{t^0}(x_0)=F(t^0, x_0)=0$ and
$$\sum_{j=1}^pc^0_j(s)\nabla f^j_{t^0}(x^0) = c^0_{p+1} x^0.$$
This means that the set $f_{t^0}^{-1}(0)$ does not intersect the sphere $\mathbb{S}^{2n-1}_{\epsilon_0}$ transversally. This is a contradiction.

\end{proof}

The previous lemma allows us to prove the following.

\begin{theorem}\label{thm4.1} With the same notation as above and suppose that the family $\{f_t\}_{t\in [0, 1]}$ has a uniform stable radius  $\epsilon_0$ for the Milnor fibration. Let $U \subset \Bbb{C}^p$ be a neighbourhood of the origin as in Lemma \ref{Lemma3.2} and let $C$ be any  (real) closed submanifold of $U\setminus (\cup_{i\in [0, 1]}D_t)$. Then, the Milnor fibrations of $f_t, t\in [0, 1]$ over $C$ are isomorphic; i.e. there is $C^\infty$-diffeomorphism
$$\Phi_t \colon f_0^{-1}(C) \cap \Bbb{B}^{2n}_{\epsilon_0} \rightarrow f_t^{-1}(C) \cap \Bbb{B}^{2n}_{\epsilon_0}, \quad t\in [0, 1],$$
which makes the following diagram commutes
\begin{equation*}
\CD f_0^{-1}(C) \cap \Bbb{B}^{2n}_{\epsilon_0} @> \Phi_t >> f_t^{-1}(C) \cap \Bbb{B}^{2n}_{\epsilon_0}  \\
@V f_0 VV @V f_t VV\\
C @> \mathrm{id} >> C
\endCD
\end{equation*}
where $\mathrm{id} $ denotes the identity map.
\end{theorem}

In order to prove the theorem, we need the following lemma.

\begin{lemma}\label{Lemma4.1}
With the assumption as in Theorem \ref{thm4.1}. Then, there exists $0<\delta(C)<\epsilon_0$   small enough such that, for any $\epsilon_0-\delta(C)<\epsilon\leq \epsilon_0$,  $t\in [0, 1]$ and  $c\in C$ the set $f_t^{-1}(c)$  intersects transversally with the sphere $\mathbb{S}^{2n-1}_{\epsilon}$.
\end{lemma}

\begin{proof}
Assume by contradiction that the conclusion of the lemma does not hold. Similarly, by Curve Selection Lemma, there are analytic curves: 
\begin{eqnarray*}
\phi \colon (0,\epsilon^{'}) \rightarrow \mathbb{B}^{2n}_{\epsilon_0}, \quad  t \colon (0,\epsilon^{'}) \rightarrow [0, 1] \quad \textrm{ and } \quad  \lambda_j \colon (0,\epsilon^{'}) \rightarrow \mathbb{C}, \ j = 1,\ldots,p+1,
\end{eqnarray*}
such that
\begin{enumerate}
\item[(a1)] $\phi(s) \rightarrow x^0\in \mathbb{S}^{2n-1}_{\epsilon_0}$ as $ s\rightarrow 0;$
\item[(a2)] $F(t(s),\phi(s))\in C$ for $s \in (0, \epsilon^{'});$
\item[(a3)] $\sum_{j=1}^p\lambda_j(s)\nabla f^j_{t(s)}(\phi(s)) = \lambda_{p+1}(s) {\phi(s)}\textrm{ for } s\in (0,\epsilon^{'});$
\item[(a4)] $\lambda_j(s), j = 1, \ldots, p + 1,$ are not all zero for $s\in(0,\epsilon^{'}).$
\end{enumerate}

By the same argument as in the proof of Lemma \ref{Lemma3.2} we may assume that there exist limits
$$c^0_j:=\lim_{s\to 0}\lambda_j(s), j=1, \ldots, p+1$$
which are not all  zero. Let $t^0:= \lim_{s\to 0}t(s)\in [0, 1]$. 

Taking the limit when $s\to 0$ in the conditions (a2) and (a3), we have $f_{t^0}(x^0)=c\in C$ and
$$\sum_{j=1}^pc^0_j\nabla f^j_{t^0}(x^0) = c^0_{p+1} x^0.$$
That means the set $f_{t^0}^{-1}(c)$ does not intersect the sphere $\mathbb{S}^{2n-1}_{\epsilon_0}$ transversally. This is a contradiction to the conclusion of Lemma \ref{Lemma3.2}.

\end{proof}

\begin{proof}[Proof of theorem \ref{thm4.1}]
Denote 
$$X:= \{(t, x)\in [0, 1]\times \Bbb{B}^{2n}_{\epsilon_0} \ : \ F(t, x)\in C\}.$$
Let $0<\delta:=\delta(C)\leq \epsilon_0$ as in Lemma \ref{Lemma4.1}. Since  for all $t\in [0, 1]$, $C$ does not intersect discriminants of $f_t$, all vector $\nabla f^1_t(x), \ldots, \nabla f^p_t(x)$ are $\Bbb{C}$-linear independent. Then, we can find a smooth map 
$$\mathbf{v}_1: U_1:= X\cap \left\{x\ :\ \|x\|<\epsilon_0-\frac{\delta}{2}\right\}\longrightarrow \Bbb{C}^n$$
such that 
\begin{eqnarray*}
\left \langle \mathbf{v}_1(t, x), {\nabla {f^j_{t}}(x)}\right \rangle = - \frac{\partial f^j_{t}}{\partial t}(x);\quad \textrm{for all}\quad j=1, \ldots, p.
\end{eqnarray*}
Similarly, by Lemma \ref{Lemma4.1}, on the set 
$$U_2:=X\cap \left \{x\ :\ \epsilon_0-\delta<\|x\|\leq \epsilon_0\right \}$$
all vectors $\nabla f^1_t(x), \ldots, \nabla f^p_t(x), x$ are $\Bbb{C}$-linear independent. Then, we can find a smooth map $\mathbf{v}_2: U_2\longrightarrow \Bbb{C}^n$ such that
\begin{itemize}
\item[(a1)] $\langle \mathbf{v}_2(t, x), {\nabla {f^j_{t}}(x)}\rangle = - \frac{\partial f^j_{t}}{\partial t}(x)$ for all $j=1, \ldots, p;$

\item[(a2)] $\langle \mathbf{v}_2(t, x), x \rangle = 0.$
\end{itemize}

Now, fix a partition of unity $\{\theta_1, \theta_2\}$ subordinated to the covering $\{U_1, U_2\}$ of $X$. We define a smooth vector field 
$$\mathbf{v}:X\longrightarrow \Bbb{C}^n$$
as $\mathbf{v}= \theta_1\mathbf{v}_1+ \theta_2\mathbf{v}_2$. We have the following:
\begin{itemize}
\item[(a3)] $\langle \mathbf{v}(t, x), {\nabla {f^j_{t}}(x)}\rangle = - \frac{\partial f^j_{t}}{\partial t}(x)$ for all $j=1, \ldots, p$ and $(t, x)\in X$;

\item[(a4)] $\langle \mathbf{v}(t, x), x \rangle = 0$ for all $(t, x)\in X$ which $\epsilon_0-\delta<\|x\|\leq \epsilon_0.$
\end{itemize}

Finally, we can see that for each $x \in f_0^{-1}(C) \cap \Bbb{B}^{2n}_{\epsilon_0},$ there exists a unique $C^\infty$-map $\varphi \colon [0, 1] \rightarrow \mathbb{C}^n$ such that
\begin{eqnarray*}
\varphi'(t) &=& \mathbf{v}(t, \varphi(t)), \quad \varphi(0) \ = \  x.
\end{eqnarray*}
Moreover, for each $t \in [0, 1],$ the map
$$\Phi_t \colon f_0^{-1}(C) \cap \Bbb{B}^{2n}_{\epsilon_0} \rightarrow f_t^{-1}(C) \cap \Bbb{B}^{2n}_{\epsilon_0}, \quad x \mapsto \varphi(t),$$
is well-defined and is a $C^\infty$-diffeomorphism, which makes the following diagram commutes
\begin{equation*}
\CD f_0^{-1}(C) \cap \Bbb{B}^{2n}_{\epsilon_0} @> \Phi_t >> f_t^{-1}(C) \cap \Bbb{B}^{2n}_{\epsilon_0}  \\
@V f_0 VV @V f_t VV\\
C @> \mathrm{id} >> C
\endCD
\end{equation*}
where $\mathrm{id} $ denotes the identity map. The proof is complete.

\end{proof}

\begin{corollary}\label{cor4.1}
	Assume that the family $\{f_t\}$ has an uniform stable radius and for each $t$, $f_t^{-1}(0)$ is a complete intersection. Then the Milnor fibers of $f_t, t\in [0, 1]$ are diffeomorphic to each other.
\end{corollary}

\begin{proof} For each $t$, the discriminant $D_t$ of $f_t$ is a hypersurface of dimension $p-1$ (see \cite[Section 2.8]{Loo1984}). Then, in any neighbourhood of the origin in $\Bbb{C}^p$, there exists a point $M$ in the complement of $\cup_{t\in [0, 1]}D_t$. Now applying Theorem \ref{thm4.1} for $C=\{M\}$, we get the conclusion.
\end{proof}

\begin{corollary}\label{cor4.2}
Let $f,g: (\Bbb{C}^n, 0)\longrightarrow (\Bbb{C}^p, 0)$ be two germs of analytic mappings such that:
\begin{itemize}
\item[i)] $\Gamma(f)=\Gamma(g)$;
\item[ii)] $f,g$ are non-degenerate and the zero sets $f^{-1}(0), g^{-1}(0)$ are complete intersections with isolated singularity at the origin.
\end{itemize}
Then $\mu_0(f)=\mu_0(g)$.
\end{corollary}

\begin{proof}
This is a direct consequence of Theorem \ref{thm4.2} and Corollary \ref{cor4.1}.
\end{proof}

The constancy of Milnor number gives us the relation on topological equivalence as bellow.

\begin{definition}{\rm
Two germs  $(X_0, 0)$ and $(X_0^{'},0)$  in $\Bbb{C}^n$ are said to be {\it topological equivalent} if there exist small neighbourhood $B, B^{'}$ of the origin in $\Bbb{C}^n$ and a homeomorphism $\phi: B\to B^{'}$ such that $\phi(X_0\cap B)= X_0^{'}\cap B^{'}.$

}
\end{definition}

\begin{corollary}
Let $f,g: (\Bbb{C}^n, 0)\longrightarrow (\Bbb{C}^p, 0)$ be two germs of analytic mappings whose Newton boundaries coincide and zero sets  $f^{-1}(0), g^{-1}(0)$ are complete intersections with isolated singularity at the origin. If $n-p\neq 2$ then $f^{-1}(0)$ and  $g^{-1}(0)$ are topological equivalent.
\end{corollary}

\begin{proof} The proof is straightforward by Theorem \ref{thm4.2}, Corollary \ref{cor4.2} and  \cite[Theorem 4.11]{Tomazella2018}.
\end{proof}

Now we will give a combinatorial formula of the Milnor number in terms of the  Newton polyhedrons. Firstly, we recall the formula by Bivia-Ausina in \cite{Bivia2007}.

\begin{theorem}\label{Biviaformula}{\rm (\cite[Theorem 3.9]{Bivia2007})} Let $f=(f^1, \ldots, f^p) : (\Bbb{C}^n, 0)\to (\Bbb{C}^p, 0)$ be an analytic mapping germ such that $f$ is convenient and  Newton-non-degenerate in the sense of Definition \ref{newtonnondegeneratedef}. Then:
$$\mu_0(f)=\nu\Big(\Gamma_{+}(f_1), \ldots, \Gamma_{+}(f_p)\Big).$$

\end{theorem}

In this section, we give a generalization for the above formula as follows.
 
\begin{theorem}\label{Milnornumberconvenient}
Let $f=(f^1, \ldots, f^p) : (\Bbb{C}^n, 0)\to (\Bbb{C}^p, 0)$ be an analytic mapping germ such that $f$ is convenient and Khovanskii non-degenerate in the sense of Definition \ref{Definition21}. Then:
$$\mu_0(f)=\nu\Big(\Gamma_{+}(f_1), \ldots, \Gamma_{+}(f_p)\Big).$$

\end{theorem}

\begin{proof}
By \cite[Lemma 6.11]{Bivia2007}, the Newton-non-degeneracy is open, then there exists a Newton-non-degenerate analytic mapping $g: (\Bbb{C}^n, 0)\to (\Bbb{C}^p, 0)$ such that $\Gamma(f)= \Gamma(g)$. It follows from \cite[Lemma 6.8, Proposition 6.9]{Bivia2007} that $g$ is Khovanskii non-degenerate. Using the same argument as in proof of Theorem \ref{thm4.2}, since $f,g$ are (Khovanskii) non-degenerate and convenient then their zero sets $f^{-1}(0), g^{-1}(0)$ are germs of complete intersections. 
Therefore, it follows from Lemma \ref{cor4.2} that  $\mu_0(f)=\mu_0(g)$. Applying Theorem \ref{Biviaformula} we get the conclusion as desired.

\end{proof}

In order to work with non-convenient maps, we will define mixed Newton number for non-convenient polytopes as below.

For a subset $\Delta\subset \Bbb{R}_{+}^n$, the Newton polyhedron $\Gamma_{+}(\Delta)$ of $\Delta$ is defined to be  the convex hull in $\Bbb{R}^n$ of union of  $\{\alpha+ \mathbb{R}^n_{+}\} $ for $\alpha\in \Delta$.

\begin{definition}
{\rm  Let $\Delta^1, \ldots, \Delta^p$ be (non-convenient) polyhedra in $\Bbb{R}^n_{+}$. The {\it mixed Newton number} of $\Delta^1, \ldots, \Delta^p$ is defined as:
$$\nu(\Delta^1, \ldots, \Delta^p)= \limsup_{\alpha^i_j\to \infty}\nu\Big(\Gamma_{+}(\Delta^1\cup \alpha^1_1\cup \ldots \cup\alpha^1_n), \ldots, \Gamma_{+}(\Delta^p\cup \alpha^p_1\cup \ldots \cup\alpha^p_n)\Big),$$
where $\alpha^i_j$ is integral point on the positive part of the $j-$th coordinate axis in $\Bbb{R}^n.$

}

\end{definition}

Remark that if the Newton polyhedra are convenient then the above notion coincides with the one defined in Definition \ref{definitionMixednumber}. The following provides a calculation for the Milnor number of non-convenient maps.

\begin{theorem}\label{Milnornumbernonconvenient}
Let $f=(f^1, \ldots, f^p) : (\Bbb{C}^n, 0)\to (\Bbb{C}^p, 0)$ be an analytic mapping germ  such that $f$ is Khovanskii non-degenerate and $f^{-1}(0)$ is a complete intersection with isolated singularity at the origin. Then:
$$\mu_0(f)=\nu\Big(\Gamma_{+}(f_1), \ldots, \Gamma_{+}(f_p)\Big).$$

\end{theorem}

\begin{proof}
It follows from Lemma \ref{lem3.3} that there is a piecewise analytic family of analytic maps 
$$\left\{\phi_t:=(f^1+t_1^1x_1^{\alpha^1_1}+\cdots + t^1_nx_n^{\alpha^1_n}, \ldots, f^p+t_1^px_1^{\alpha^p_1}+\cdots + t^p_nx_n^{\alpha^p_n})\right\}$$
 having an uniform stable radius, which zero set of each map is a complete intersection with isolated singularity at the origin and $\phi_0=f$, where the exponent $\alpha^i_j$ is large enough integral points on the positive part of the $j-$th coordinate axis of $\Bbb{R}^n.$ Then, according to the Corollary \ref{cor4.1}, for all parameter $t$:
 $$\mu_0(f)=\mu_0(\phi_t).$$
 Also, by Lemma \ref{lem3.3} for generic $t$, the map $\phi_t$ is (Khovanskii) non-degenerate and convenient. Therefore, by Theorem \ref{Milnornumberconvenient}, we have:
 $$\mu_0(\phi_t)=\nu\Big(\Gamma_{+}(\phi_t)\Big)=\nu(\Gamma_{+}(f)).$$
The proof is complete.
\end{proof}


\begin{thebibliography}{10}

\bibitem{Bivia2007}
  C. Bivia-Ausina. 
  \newblock Mixed Newton numbers and isolated complete intersection singularities, \newblock {\em Proc. Lond. Math. Soc.}, 94(3): 749–771, 2007.

\bibitem{Briancon}
J. Briançon.
\newblock Le théorème de Kouchnirenko, unpublished lecture note.


\bibitem{Eyral2017}
C. Eyral.
\newblock Uniform stable radius, Le numbers and topological triviality for line singularities,
\newblock {\em Pacific J. Math.}, 291(2): 359-367,2017.


\bibitem{Eyral-Oka2017}
C. Eyral and M. Oka.
\newblock Non-compact Newton boundary and Whitney equisingularity for non-isolated singularities,
\newblock {\em Advances in Math.}, 316: 94-113,2017.



\bibitem{Hamm1971}
H. A. Hamm.
\newblock  Lokale topologische Eigenschaften komplexer Raume.
\newblock {\em Math. Ann.}, 191: 235–252,  1971.



\bibitem{HaHV1997}
H.~V. H\`a and T.~S. Ph\d{a}m.
\newblock Invariance of the global monodromies in families of polynomials of
  two complex variables.
\newblock {\em Acta Math. Vietnam.}, 22(2):515--526, 1997.





\bibitem{HaHV1996}
H.~V. H\`a and A.~Zaharia.
\newblock Families of polynomials with total {{M}}ilnor number constant.
\newblock {\em Math. Ann.}, 313:481--488, 1996.


\bibitem{Huybr}
D. Huybrechts. {\em Complex Geometry: An Introduction}. Springer, 2005.

\bibitem{KK}
K. Kaveh and A.G. Khovanskii.
\newblock On mixed multiplicities of ideals, arXiv:1310.7979.

\bibitem{KN}
D. Kerner and A. Nemethi. 
\newblock Durfee-type bound for some non-degenerate complete intersection singularities
\newblock {\em Math. Z.}, 285: 159-175, 2017. 


\bibitem{Kouchnirenko1976}
A.~G. Kouchnirenko.
\newblock Polyhedres de {{N}}ewton et nombre de {{M}}ilnor.
\newblock {\em Invent. Math.}, 32:1--31, 1976.

\bibitem{Le1976}
 D.T. Le and C.P. Ramanujam. 
\newblock  Invariance of Milnors number implies the invariance of topological type. 
\newblock {\em Amer. J. Math.} 98: 67–78, 1976.


\bibitem{Loo1984}
E.J.N.~Looijenga.
\newblock {\em Isolated singular points on complete intersections}, {\em London Mathematical Society lecture note series \bf{77}}.
\newblock Cambridge University Press London-New York, 1984.





\bibitem{Milnor1968}
J.~Milnor.
\newblock {\em Singular points of complex hypersurfaces}, {\em Annals of Mathematics Studies \bf{61}}.
\newblock Princeton University press, 1968.


\bibitem{Thang2019}
T.T.Nguyen, P.~P. Ph\d{a}m and T.~S. Ph\d{a}m. 
\newblock Bifurcation sets and global monodromies of Newton nondegenerate polynomials on algebraic sets.
\newblock  {\em Publ. RIMS Kyoto Univ.}, 55: 1–24, 2019.


\bibitem{Tomazella2018}
 J. J. Nuno-Ballesteros, B. Orefice-Okamoto, J. N. Tomazella. 
\newblock Equisingularity of families of isolated determinantal singularities. 
\newblock {\em Math. Z.} 289, no. 3-4: 1409–1425, 2018. 


\bibitem{Oka1973}
M. Oka.
\newblock “Deformation of Milnor fiberings”, 
\newblock {\em J. Fac. Sci. Univ. Tokyo Sect. IA Math.} 20 (1973), 397–400. Correction in 27:2 (1980), 463–464.


\bibitem{Oka1979}
M. Oka.
\newblock On the bifurcation of the multiplicity and topology of the Newton boundary.
\newblock {\em J. Math. Soc. Japan}, 31: 435–450, 1979.


\bibitem{Oka1990}
M. Oka. 
\newblock Principal zeta-function of non-degenerate complete intersection singularity.
\newblock  {\em J. Fac. Sci. Univ. Tokyo}, 37: 11–32, 1990.


\bibitem{Oka1997}
M. Oka.
\newblock {\em Non-degenerate complete intersection singularity}, {\em Actualit'es Math'ematiques}, 
\newblock  Hermann, Paris, 1997. 



\bibitem{Phamts2008}
T.~S. Ph\d{a}m.
\newblock On the topology of the {{N}}ewton boundary at infinity.
\newblock {\em J. Math. Soc. Japan}, 60(4):1065--1081, 2008.

\bibitem{Phamts2010}
T.~S. Ph\d{a}m.
\newblock Invariance of the global monodromies in families of nondegenerate
  polynomials in two variables.
\newblock {\em Kodai Math. J.}, 33(2):294--309, 2010.


\end{thebibliography}

\end{document}